\documentclass[11pt]{amsart}
\usepackage{amsthm,amsmath,amssymb}

\newcommand{\C}{{\mathbb{C}}}
\newcommand{\R}{{\mathbb{R}}}
\newcommand{\sm}{{\operatorname{Sym}}}

\newtheorem{thm}{Theorem}

\theoremstyle{definition}
\newtheorem{rem}[thm]{Remark}

\begin{document}

\title{Eigenvalues of real symmetric matrices}
\markright{Eigenvalues of real symmetric matrices}
\author{Meinolf Geck}

\date{}

\begin{abstract} We present a proof of the existence of real eigenvalues
of real symmetric matrices which does not rely on any limit or compactness
arguments, but only uses the notions of ''sup'', ''inf''.
\end{abstract}

\maketitle

%%%%%%%%%%%%%%%%%%%%%%%%%%%%%%%%%%%%%%%%%%%%%%%%%%%%%%%%%%%%%%%%%%%%%%%
Let $M_n(\R)$ denote the set of all real $n\times n$-matrices and 
$\sm_n(\R)$ be the set of all $A\in M_n(\R)$ that are symmetric. A key 
ingredient of the ''Spectral Theorem'' is the existence of a real eigenvalue 
of a matrix $A\in \sm_n(\R)$. In some way, this uses limit or compactness 
arguments in $\R^n$ (e.g., \cite[Kap.~6, \S 2]{Koe}, \cite{wulf}) or the 
fact that $\C$ is algebraically closed (e.g., \cite[\S 6.4]{fis}). 
Usually, none of these are available in a first course on linear algebra; 
in any case, it seems desirable to isolate the bare ''analytic'' 
prerequisites of this basic result about matrices. We present here a
slight variation of the argument in \cite{Koe}, which refers at only one 
place to the completeness axiom for $\R$.

For $v,w\in \R^n$ (column vectors) we let $\langle v,w\rangle:={^tv}\cdot w$ 
denote the usual scalar product (${^tv}$ is the transpose of $v$, that 
is, a row vector). The Euclidean norm of $v$ is denoted by $\lVert v\rVert=
\sqrt{\langle v, v\rangle}$. We define the norm of a matrix 
$A=(a_{ij})\in M_n(\R)$ by $|A|_\infty= \max\{|a_{ij}|\colon 1\leq i,j 
\leq n\}$. All we need to know about these norms is the following inequality:
\begin{equation*}
\lVert A\cdot v\rVert \leq \sqrt{n}^3\,|A|_\infty \lVert v\rVert \qquad 
\mbox{for all $v\in \R^n$}.  \tag{$\dagger$}
\end{equation*}
This easily follows from the inequalities $|w|_\infty\leq \lVert w\rVert
\leq \sqrt{n}|w|_\infty$ and $|A\cdot w|_\infty \leq n|A|_\infty |w|_\infty$; 
we set $|w|_\infty =\max\{|w_1|,\ldots,|w_n|\}$ for any $w={^t(w_1,
\ldots,w_n)}\in \R^n$.

\begin{rem} \label{rem2} Let $A\in \sm_n(\R)$. By a finite sequence of 
row and column operations, $A$ can be transformed by congruence into a 
diagonal matrix. That is, there is a nonsingular real matrix $P$ and a 
real diagonal matrix $D$ such that $A={^tP}\cdot D\cdot P$; e.g.,
\cite[\S 6.7]{fis}. Since positive real numbers have square roots, we can 
further assume that all non-zero diagonal entries of $D$ are $\pm 1$. Now 
assume that $A\succeq 0$, that is, $\langle v,A\cdot v\rangle \geq 0$ for 
all $v\in \R^n$. (Such a matrix is called positive semidefinite.) Then 
all non-zero diagonal entries of $D$ must be $+1$. Consequently, we have 
the implication:
\[ A\succeq 0 \;\mbox{ and }\; \det(A)\neq 0\quad\Rightarrow\quad
A={^tP}\cdot P\quad \mbox{with $P\in M_n(\R)$ invertible}.\]
\end{rem}

\begin{rem} \label{rem4} Let $A=(a_{ij}) \in\sm_n(\R)$. If $v\in 
\R^n$ is such that $\lVert v\rVert=1$, then all components of $v$ have 
absolute value $\leq 1$ and so $|\langle v,A\cdot v\rangle|\leq 
\sum_{i,j=1}^n |a_{ij}|$. Hence, the set
\[S(A):=\{ \langle v,A\cdot v\rangle \colon v\in \R^n,\rVert v\rVert=1\}
\subseteq \R\]
is bounded. In particular, this set has a greatest lower bound $\mu(A)=
\inf S(A)$. We have
\[ \langle v,A\cdot v\rangle \geq \mu(A)\lVert v\rVert^2 \qquad \mbox{for 
all $v\in \R^n$}.\]
This inequality is clear if $v=0$; if $v\neq 0$, then set $w:=v/\lVert 
v\rVert$ and note that $\langle w,A\cdot w\rangle \geq \mu(A)$.
\end{rem}

By a limit or compactness argument, one can deduce that there exists a 
vector $v_0\in\R^n$ such that $\lVert v_0\rVert=1$ and $\langle v_0,A\cdot 
v_0 \rangle=\mu(A)$. It then follows easily that $v_0$ is an eigenvector 
of $A$ with eigenvalue $\mu(A)$ (see \cite[Kap.~6, \S 2, no.~4]{Koe}). The 
proof below avoids this line of reasoning.

\begin{thm} \label{thm1} If $A\in \sm_n(\R)$, then $\mu(A)$ is an eigenvalue
of $A$.
\end{thm}

\begin{proof} Let $I_n$ be the identity matrix and set $B:=A-\mu(A)
I_n\in \sm_n(\R)$. If $\det(B)=0$, then $\mu(A)$ is an eigenvalue of $A$. So 
let us now assume that $\det(B)\neq 0$. We have $\langle v, B
\cdot v\rangle= \langle v,A\cdot v\rangle-\mu(A)\lVert v\rVert^2$ for all 
$v\in\R^{n}$. Remark~\ref{rem4} shows that $B\succeq 0$ and $\mu(B)=
\inf S(B)=0$. Since also $\det(B)\neq 0$, we can write $B={^tP}\cdot P$, 
where $P\in M_n(\R)$ is invertible (see Remark~\ref{rem2}). 

Now, for any $v\in \R^n$, we have $\langle v,B\cdot v\rangle={^tv} \cdot 
B \cdot v = \lVert P\cdot v\rVert^2$. Furthermore, if $\lVert v\rVert=1$, 
then $1=\lVert P^{-1} \cdot (P\cdot v)\rVert \leq \sqrt{n}^3 |P^{-1}|_\infty 
\lVert P\cdot v \rVert$, using ($\dagger$). Thus, $\langle v,B \cdot v\rangle
\geq 1/(n^3|P^{-1}|_\infty^2)>0$ for all $v\in\R^n$ such that $\lVert 
v\rVert=1$, contradicting $\inf S(B)=0$. 
\end{proof}

\begin{rem} \label{rem5} The argument also works for Hermitian matrices 
$A\in M_n(\C)$. One just has to use the Hermitian product $\langle v,w
\rangle ={^t\overline{v}} \cdot w$ for $v,w\in\C^n$, where the bar denotes 
complex conjugation. If $A$ is Hermitian, then $\langle v,A\cdot v\rangle
\in \R$ for all $v\in \C^n$, so we can define $\mu(A)= \inf S(A)$ as above.
\end{rem}

%%%%%%%%%%%%%%%%%%%%%%%%%%%%%%%%%%%%%%%%%%%%%%%%%%%%%%%%%%%%%%%%%%%%%%%

\bigskip
\noindent \itshape{IAZ -- Lehrstuhl f\"ur Algebra, Universit\"at Stuttgart,
Pfaffenwaldring 57, 70569 Stuttgart, Germany\\
meinolf.geck@mathematik.uni-stuttgart.de}


\begin{thebibliography}{131}
\bibitem{fis}
S. H. Friedberg, A. J. Insel, L. E. Spence, \textit{Linear Algebra}, 3rd ed.,
Prentice Hall, NJ, 1997.

\bibitem{Koe}
M. Koecher, \textit{Lineare Algebra und analytische Geometrie}, Grundwissen 
Mathematik, Springer-Verlag, Berlin, 1983.

\bibitem{wulf}
H. S. Wulf, An algorithm-inspired proof of the spectral theorem in 
$E^n$, {\it Amer. Math. Monthly} {\bf 88} (1981) 49--50.
\end{thebibliography}
\end{document}